\documentclass[12pt,a4paper,reqno]{amsart}

\pdfoutput=1
\usepackage{latexsym}
\usepackage{amsmath}
\usepackage{amsfonts}
\usepackage{a4wide}
\usepackage{amssymb}
\usepackage{amsthm}
\usepackage[colorlinks, citecolor=blue, linkcolor=red]{hyperref}
\usepackage{cleveref}
\usepackage{tensor}
\usepackage{cancel}
\usepackage{aligned-overset}
\usepackage[isbn=false,url=false,doi=false,giveninits=true,date=year]{biblatex}
\renewbibmacro{in:}{}
\usepackage{pxfonts}
\usepackage{pdfpages}

\newcommand{\mint}[1]{\int_{M} {#1} \ dV_{g}}

\newcommand{\mmmint}[1]{\int_{M} \left\{{#1}\right\} dV_{g}}

\theoremstyle{plain}
\newtheorem{theorem}{Theorem}
\newtheorem{proposition}{Proposition}
\newtheorem{corollary}{Corollary}

\newcommand{\remark}[1]{{\noindent\textbf{Remark.} {#1}}}
\newcommand{\thref}[1]{Theorem \hyperref[#1]{\ref*{#1}}}
\newcommand{\propref}[1]{Proposition \hyperref[#1]{\ref*{#1}}}
\newcommand{\corref}[1]{Corollary \hyperref[#1]{\ref*{#1}}}
\newcommand{\lemref}[1]{Lemma \hyperref[#1]{\ref*{#1}}}

\title{New Rigidity Results for critical Metrics of some quadratic curvature functionals}

\author{Marco Bernardini}

\DeclareFieldFormat{pages}{#1}

\DeclareFieldFormat[article,periodical]{number}{\bibstring{number}~#1}

\DeclareFieldFormat[article,periodical]{volume}{\mkbibbold{#1}}

\renewbibmacro*{journal+issuetitle}{%
	\usebibmacro{journal}%
	\setunit*{\addspace}%
	\iffieldundef{series}
	{}
	{\newunit
		\printfield{series}%
		\setunit{\addspace}}%
	\printfield{volume}%
	\setunit{\addspace}%
	\usebibmacro{issue+date}%
	\setunit{\addcomma\space}%
	\printfield{number}%
	\setunit{\addcolon\space}%
	\usebibmacro{issue}%
	\setunit{\addcomma\space}%
	\printfield{eid}
	\newunit}

\begin{document}
	
	\begin{abstract}
		We prove a new rigidity result for metrics defined on closed smooth $ n $-manifolds that are critical for the quadratic functional $ \mathfrak{F}_{t} $, which depends on the Ricci curvature $ Ric $ and the scalar curvature $ R $, and that satisfy a pinching condition of the form $ Sec > \epsilon R $, where $ \epsilon $ is a function of $ t $ and $ n $, while $ Sec $ denotes the sectional curvature. In particular, we show that Bach-flat metrics with constant scalar curvature satisfying $ Sec > \frac{1}{48} R $ are Einstein and, by a known result, are isometric to $ \mathbb{S}^{4} $, $ \mathbb{RP}^{4} $ or $ \mathbb{CP}^{2} $.
	\end{abstract}
	
	\maketitle
	
	\begin{center}
		\medskip
		\centerline{\bf AMS subject classification:  53C24, 53C25}
	\end{center}
	
	\section{Introduction}
	
	Given $ M^{n} $ closed smooth manifold, we denote with $ \mathcal{M}_{1}(M^{n}) $ the space of smooth Riemannian metrics yielding unitary volume on $ M^{n} $. Given a Riemannian metric $ g $ on $ M^{n} $, let $ Riem $, $ W $, $ Ric $, $ R $ be the corresponding Riemann, Weyl, Ricci and scalar curvatures, respectively.
	\\Among all the metrics $ g \in \mathcal{M}_{1}(M^{n}) $, we can isolate "special" ones by looking among the critical metrics of some curvature functional. Indeed, it is well known \cite{besse_einstein_1987} that Einstein metrics, metrics for which $ Ric = \lambda g $ with $ \lambda \in \mathbb{R} $, arise as critical points of the Einstein-Hilbert functional:
	\[ \mathfrak{S}(g) = \mint{R_{g}} \]
	Since Einstein metrics arise naturally as critical points of such a basic geometric functional, it is legitimate to investigate more generic curvature functionals and understand when it is still the case.
	\\A basis of quadratic curvature functionals is given by:
	\[ \mathfrak{W} = \mint{|W|^{2}} \qquad \mathfrak{r} = \mint{|Ric|^{2}} \qquad \mathfrak{S}_{2} = \mint{R^{2}} \]
	and, via the Riemann tensor decomposition, we obtain:
	\begin{align*}
		\mathfrak{R} = \mint{|Riem|^{2}} &= \mmmint{|W|^{2} + \frac{4}{n-2} |Ric|^{2} - \frac{2}{(n-1)(n-2)} R^{2}} \\
		&= \mathfrak{W} + \frac{4}{n-2} \mathfrak{r} - \frac{2}{(n-1)(n-2)} \mathfrak{S}_{2}
	\end{align*}
	If $ n = 4 $, the Chern-Gauss-Bonnet formula holds:
	\begin{equation}\label{eq:ChernGaussBonnet}
		\mmmint{|W|^{2} - 2 |Ric|^{2} + \frac{2}{3} R^{2}} = 32 \pi^{2} \scalebox{1.3}{$\chi$}(M)
	\end{equation}
	which tells us that $ \mathfrak{W} $ is a linear combination of $ \mathfrak{r} $, $ \mathfrak{S}_{2} $ and the topological term $ \scalebox{1.3}{$\chi$}(M) $, which is the Euler characteristic of $ M $.
	\\The study of quadratic Riemannian functionals was initiated by Berger \cite{berger_quelques_1970}, see \cite[Chapter 4]{besse_einstein_1987} for a survey.
	\\Since for $ n > 4 $ Einstein metrics are not generally critical points of the $ \mathfrak{W} $ functional on $ \mathcal{M}_{1}(M^{n}) $, we will focus only on the following curvature functional:
	\[ \mathfrak{F}_{t} = \mint{|Ric|^{2}} + t \mint{R^{2}} = \mathfrak{r} + t \mathfrak{S}_{2} \]
	with $ t \in \mathbb{R} $ and where $ t = - \infty $ formally corresponds to the curvature functional $ \mathfrak{S}_{2} $. Being $ \mathfrak{F}_{t} $ not scale-invariant for $ n > 4 $, it is natural to study it on $ \mathcal{M}_{1}(M^{n}) $ (or to properly renormalize it by the volume of the manifold). Such functional was first introduced by Berger in \cite{berger_quelques_1970} and later on studied by many, see \cite{besse_einstein_1987, gursky_rigidity_2015, catino_rigidity_2015}.
	\\While every Einstein metric is critical for $ \mathfrak{F}_{t} $, as already observed in \cite{besse_einstein_1987}, the converse generally fails. As an example, in case $ n = 4 $ and $ t = - \frac{1}{3} $, Bach-flat metrics are the critical metrics for $ \mathfrak{F}_{t} $, but generally these are not Einstein metrics (for further counter examples see \cite[Chapter 4]{besse_einstein_1987} and \cite{lamontagne_critical_1993}). Hence, it is natural to understand under which conditions a critical metric for $ \mathfrak{F}_{t} $ is, infact, Einstein.
	\\Such additional constraints usually take the shape of some pointwise or, more weakly, of some integral-type restriction of the critical metric's curvature. For instance, in \cite[Proposition 1.1]{anderson_extrema_1997} the scalar curvature is assumed to have definite sign (which actually holds for any dimension by \cite[Proposition 3.1]{catino_critical_2014}) for critical metrics on $ \mathcal{M}_{1}(M^{3}) $ of $ \mathfrak{S}_{2} $; in \cite{lamontagne_remarque_1994} it is proved that every critical metric for $ \mathfrak{F}_{-1/3} $ on $ \mathcal{M}_{1}(M^{4}) $ (variationally equivalent to $ \mathfrak{S}_{2} $) with non-positive sectional curvature is Einstein; for $ \mathfrak{F}_{-1/3}$ on $ \mathcal{M}_{1}(M^{3}) $ in \cite{tanno_deformations_1975} a pointwise pinching condition on the Ricci curvature is assumed; for $ \mathfrak{F}_{-3/8} $ on $ \mathcal{M}_{1}(M^{3}) $ in \cite{gursky_new_2001} it is proved that every critical metric must be Einstein, hence a space form, just by assuming an integral condition, namely $ \mathfrak{F}_{-3/8} \leq 0 $ (this result was extended in \cite{hu_new_2004} for dimension greater than four in the locally conformally flat case); finally, in \cite{catino_rigidity_2015} the author assumed nonnegative sectional curvature in order to prove that critical metrics for $ \mathfrak{F}_{t}$ on $ \mathcal{M}_{1}(M^{n}) $ are Einstein, provided $ t < - \frac{1}{2} $ (in particular, \thref{th:Io_1} below extends  \cite[Theorem 1.1]{catino_rigidity_2015}). Indeed, it is well known that compact Einstein manifolds can be classified, provided they are enough positively curved. Infact, classification results are obtained in \cite{brendle_einstein_2010} by assuming nonnegative isotropic curvature; in \cite{berger_quelques_1970} via the assumption of a weak $ \frac{1}{4} $-pinching condition on the curvature (that in dimension four is implied by the assumption $ Sec \geq \frac{1}{24} R $); in dimension four $ \mathbb{S}^{4} $, $ \mathbb{RP}^{4} $ or $ \mathbb{CP}^{2} $ are isolated in \cite{yang_rigidity_2000} by requiring $ Sec \geq \epsilon R $ with $ \epsilon = \frac{\sqrt{1249} - 23}{480} $. Such lower bound was improved up to $ \epsilon = \frac{1}{48} $ in \cite{ribeiro_rigidity_2016} and is conjectured in \cite{yang_rigidity_2000} to be improvable up to $ \epsilon = 0 $.
	\\ \\In this paper we present new rigidity results regarding critical metrics for $ \mathfrak{F}_{t} $ on $ \mathcal{M}_{1}(M^{n}) $.
	\\We have found that:
	\begin{theorem}\label{th:Io_1}
		Let $ M^{n} $ closed smooth n-manifold with $ n \geq 3 $, let $ g $ be a critical metric for $ \mathfrak{F}_{t} $ on $ \mathcal{M}_{1}(M^{n}) $ with $ t \leq  - \frac{1}{2} $, such that $ R \geq 0 $ and $ Sec > \frac{1 + 2 t}{(n-2)^{2}} R $. Then, $ g $ is Einstein.
	\end{theorem}
	\begin{theorem}\label{th:Io_2}
		Let $ M^{n} $ closed smooth n-manifold with $ n \geq 3 $, let $ g $ be a critical metric for $ \mathfrak{F}_{t} $ on $ \mathcal{M}_{1}(M^{n}) $ with $ t > - \frac{1}{2} $, such that $ R = const \geq 0 $ and $ Sec > \frac{1 + 2t}{n^{2} - n + 4} R $. Then, $ g $ is Einstein.
	\end{theorem}
	By a direct application of \thref{th:Io_2}, we have the following result in dimension four and $ t = - \frac{1}{3} $, for which Bach-flat metrics are the critical points of $ \mathfrak{F}_{- \frac{1}{3}} $:
	\begin{corollary}\label{th:Io_3}
		Let $ M^{4} $ closed smooth $ 4 $-manifold, let $ g $ be a Bach-flat metric, such that $ R = const \geq 0 $ and $ Sec > \frac{1}{48} R $. Then, $ g $ is Einstein. Moreover, $ (M^{4},g) $ is isometric to $ \mathbb{S}^{4} $, $ \mathbb{RP}^{4} $ or $ \mathbb{CP}^{2} $ with standard metrics.
	\end{corollary}
	Where the classification results from the application of \cite[Theorem 1.3]{ribeiro_rigidity_2016}, of which we recover the optimal pinching value of $ \epsilon = \frac{1}{48} $ on the sectional curvature.
	\\Whereas, for the case of $ n = 4 $ and $ t \neq - \frac{1}{3} $, we have:
	\begin{corollary}\label{th:Io_5}
		Let $ M^{4} $ closed smooth $ 4 $-manifold, let $ g $ be a critical metric for $ \mathfrak{F}_{t} $ on $ \mathcal{M}_{1}(M^{4}) $ with $ t \neq - \frac{1}{3} $, such that $ R \geq 0 $ and $ Sec > \epsilon(t) R $, where:
		\begin{equation*}
			\epsilon(t) = \begin{cases}
				\frac{1 + 2 t}{4} \quad &t \leq - \frac{1}{2} \\
				\frac{1 + 2 t}{16} \quad &t > - \frac{1}{2} \land t \neq - \frac{1}{3}
			\end{cases}
		\end{equation*}
		Then, $ g $ is Einstein.
	\end{corollary}
	\begin{proof}
		Since for $ n = 4 $ and $ t \neq - \frac{1}{3} $ critical metrics for $ \mathfrak{F}_{t} $ have constant scalar curvature (see \corref{th:Rconst}), in the $ t \leq - \frac{1}{2} $ range we can apply \thref{th:Io_1} and in the $ t > - \frac{1}{2} \ \land \ t \neq - \frac{1}{3} $ range, being the assumption $ R = const $ automatically satisfied, we can apply \thref{th:Io_2}.
	\end{proof}
	\remark{We notice that, as $ t \to - \infty $, the requirement $ Sec \geq \epsilon R $ is vacuously satisfied and, formally, $ \mathfrak{F}_{t} \to \mathfrak{S}_{2} $. Hence, \corref{th:Io_5} is in agreement with the fact that compact critical metrics for $ \mathfrak{S}_{2} $ on $ \mathcal{M}_{1} (M^{n}) $ with nonnegative scalar curvature are either scalar-flat or Einstein (see \cite[Proposition 5.1]{catino_perspective_2020}).\\ \\}
	\remark{Arguing as in \cite[Theorem 1.2]{catino_rigidity_2015}, it might be possible to include in the above rigidity results also the limit case of $ Sec = \epsilon R $.}
	\
	
	\section{The Euler-Lagrange equations and some Lemmas}
	
	This section is dedicated to the computation of the Euler-Lagrange equations of the $ \mathfrak{F}_{t} $ functional, which are equations satisfied by the corresponding critical metrics.
	\\We fix the notation by recalling that the Riemann curvature (3,1)-tensor $ Riem $ of a Riemannian manifold $ (M^{n},g) $ is defined, as in \cite{lee_introduction_2018}, by:
	\[
	Riem(X,Y)Z = \nabla_{X} \nabla_{Y} Z - \nabla_{Y} \nabla_{X} Z - \nabla_{[X,Y]} Z
	\]
	Which, in local coordinates, reads $ \tensor{R}{^{l}_{i}_{j}_{k}} \frac{\partial}{\partial x^{l}} = Riem(\frac{\partial}{\partial x^{i}},\frac{\partial}{\partial x^{j}}) \frac{\partial}{\partial x^{k}} $. Instead, its (4,0)-version is given by $ R_{lijk} = g_{lp} \tensor{R}{^{p}_{i}_{j}_{k}} $. Where, as for the rest of the entire paper, the Einstein summation convention is adopted.
	\\The Ricci tensor $ Ric $ is given by contraction of the Riemann curvature as $ (Ric)_{ik} = R_{ik} = g^{lj} R_{lijk} $. By contracting once more, we obtain the scalar curvature $ R = g^{ik} R_{ik} $. Finally, the traceless Ricci curvature $ \mathring{Ric} $ is given by $ \mathring{R}_{ik} = R_{ik} - \frac{1}{n} R g_{ik} $.
	\\The decomposition of the Riemann curvature holds:
	\[ Riem = W + \frac{1}{n-2} \mathring{Ric} \circledwedge g + \frac{1}{2n(n-1)}R g \circledwedge g \]
	where $ W $ is the Weyl tensor and $ \circledwedge $ is the Kulkarni-Nomizu product.
	\\In order to compute the Euler-Lagrange equations of the $ \mathfrak{F}_{t} $ functional, we first compute its gradient. Following the computations in \cite[Proposition 4.66]{besse_einstein_1987}, we get that the gradients of the $ \mathfrak{r} $ and the $ \mathfrak{S}_{2} $ functionals are, respectively:
	\begin{align*}
		\nabla \mathfrak{r} &= - \Delta Ric - 2 R_{ikjl} R^{kl} - \frac{1}{2} \Delta R g + \frac{1}{2} |Ric|^{2} g + Hess R \\
		\nabla \mathfrak{S}_{2} &= -2 R_{g} Ric_{g} + \frac{1}{2} R_{g}^{2} g + 2 Hess R_{g} -2 (\Delta R_{g}) g
	\end{align*}
	As a conseguence, the gradient of $ \mathfrak{F}_{t} $ is:
	\[ 	(\nabla \mathfrak{F}_{t})_{ij} = (\nabla \mathfrak{r})_{ij} + t (\nabla \mathfrak{S}_{2})_{ij} = - \Delta R_{ij} + (1 +2t) \nabla^{2}_{i,j} R - \frac{1 + 4t}{2} \Delta R g_{ij} + \frac{1}{2} (|Ric|^{2} + t R^{2}) g_{ij} - 2 R_{ikjl} R_{kl} - 2 t R R_{ij} \]
	Hence, since we are carrying out the optimization on $ \mathcal{M}_{1} (M^{n}) $, a metric $ g $ will be critical for $ \mathfrak{F}_{t} $ if and only if (see \cite{besse_einstein_1987}) $ \nabla \mathfrak{F}_{t} = c g $ for some Lagrange multiplier $ c \in \mathbb{R} $. Then, the trace of such equation reads:
	\[
	\frac{n-4}{2} (|Ric|^{2} + t R^{2}) - \frac{n + 4(n-1)t}{2} \Delta R = n c
	\]
	By putting these two equations together we get:
	\[
	- \Delta R_{ij} - 2 R_{ikjl} R_{kl} + (1+2t) \nabla^{2}_{i,j} R - \frac{2t}{n} \Delta R g_{ij} + \frac{2}{n} (|Ric|^{2} + t R^{2}) g_{ij} - 2t R R_{ij} = 0
	\]
	and:
	\[ (n + 4(n-1)t) \Delta R = (n-4) (|Ric|^{2} + t R^{2} - \lambda) \]
	with $ \lambda = \mathfrak{F}_{t}(g) $.
	Hence, the Euler-Lagrange equations for $ \mathfrak{F}_{t} $ read as:
	\begin{proposition}
		Let $ M^{n} $ be a closed manifold of dimension $ n \geq 3 $. A metric $ g $ is critical for $ \mathfrak{F}_{t} $ on $ \mathcal{M}_{1}(M^{n}) $ if and only if it satisfies the following equations:
		\begin{equation}\label{eq:EL_Ft}
			\begin{cases}
				\Delta \mathring{R}_{ij} = (1 + 2t) \nabla^{2}_{i,j} R - \frac{1 + 2t}{n} \Delta R g_{ij} - 2 R_{ikjl} \mathring{R}_{kl} - \frac{2 + 2 nt}{n} R \mathring{R}_{ij} + \frac{2}{n} |\mathring{Ric}|^{2} g_{ij} \\
				(n + 4(n-1)t) \Delta R = (n-4) (|Ric|^{2} + t R^{2} - \lambda)
			\end{cases}
		\end{equation}
		with $ \lambda = \mathfrak{F}_{t}(g) $.
	\end{proposition}
	It directly follows that (see \cite[Corollary 4.67]{besse_einstein_1987}):
	\begin{corollary}
		Any Einstein metric is critical for $ \mathfrak{F}_{t} $ on $ \mathcal{M}_{1}(M^{n}) $.
	\end{corollary}
	In case of $ n = 4 $ and $ t \neq - \frac{1}{3} $, we can immediately infer from the second equation of the system \eqref{eq:EL_Ft} the following property:
	\begin{corollary}\label{th:Rconst}
		Let $ M^{4} $ closed 4-manifold. If $ g $ is a critical metric for $ \mathfrak{F}_{t} $ on $ \mathcal{M}_{1}(M^{4}) $ for some $ t \neq - \frac{1}{3} $, then $ g $ has constant scalar curvature.
	\end{corollary}
	This generally will fail in case $ t = - \frac{1}{3} $, since in dimension four, via Chern-Gauss-Bonnet formula \eqref{eq:ChernGaussBonnet}, $ \mathfrak{F}_{- \frac{1}{3}} $ is variationally equivalent to the Weyl functional $ \mathfrak{W} $, whose critical metrics are Bach-flat and thus with non-constant scalar curvature, in general.
	\\By contracting the first equation of the system \eqref{eq:EL_Ft} with $ \mathring{R}_{ij} $, we obtain the following Weitzenb\"ock formula and its integral version:
	\begin{proposition}\label{th:Cat_5.16}
		Let $ M^{n} $ closed manifold, if $ g $ is critical for $ \mathfrak{F}_{t} $ on $ \mathcal{M}_{1}(M^{n}) $, then:
		\begin{equation*}\label{eq:structure2}
			\frac{1}{2} \Delta |\mathring{Ric}|^{2} = |\nabla \mathring{Ric}|^{2} + (1+2t) \mathring{R}_{ij} \nabla^{2}_{i,j} R - 2 R_{ikjl} \mathring{R}_{ij} \mathring{R}_{kl} - \frac{2 + 2nt}{n} R |\mathring{Ric}|^{2}
		\end{equation*}
		and:
		\[ \mmmint{|\nabla \mathring{Ric}|^{2} - \frac{(n-2)(1 + 2t)}{2n} |\nabla R|^{2}} = 2 \mmmint{R_{ikjl} \mathring{R}_{ij} \mathring{R}_{kl} + \frac{1 + nt}{n} R |\mathring{Ric}|^{2}} \]
	\end{proposition}
	Another key tool in the proof of the main result is the following pointwise estimate available for the metrics that satisfy $ Sec \geq \epsilon R $ for some $ \epsilon \in \mathbb{R} $.
	\begin{proposition}{(\cite[Proposition 2.1]{catino_rigidity_2019})}\label{th:sectional}
		Let $ (M^{n},g) $ Riemannian manifold with $ n \geq 3 $, if $ Sec \geq \epsilon R $ for some $ \epsilon \in \mathbb{R} $, then:
		\begin{gather*}
			R_{ijkl} \mathring{R}_{ik} \mathring{R}_{jl} \leq \frac{1 - n^{2} \epsilon}{n} R |\mathring{Ric}|^{2} + \mathring{R}_{ij} \mathring{R}_{ik} \mathring{R}_{jk} \\
			R_{ijkl} \mathring{R}_{ik} \mathring{R}_{jl} \leq \frac{n^{2} - 4 n + 2 - n^{2} (n-2)(n-3) \epsilon}{2 n} R |\mathring{Ric}|^{2} - (n-1) \mathring{R}_{ij} \mathring{R}_{ik} \mathring{R}_{jk}
		\end{gather*}
	\end{proposition}
	By taking the convex combination of such inequalities, we get the following.
	\begin{corollary}\label{th:Cat_5.18}
		Let $ (M^{n},g) $ Riemannian manifold, with $ n \geq 3 $. If $ Sec \geq \epsilon R $ for some $ \epsilon \in \mathbb{R} $, then for any $ \ s \in [0,1] $:
		\begin{multline}\label{eq:sec}
			R_{ijkl} \mathring{R}_{ik} \mathring{R}_{jl} \leq \left( \frac{n^{2} - 4n + 2 - n^{2} (n-2) (n-3) \epsilon}{2n} - \frac{n-4}{2} (1 - n(n-1) \epsilon) s \right) R |\mathring{Ric}|^{2} \\
			- (n - 1 - ns) \mathring{R}_{ij} \mathring{R}_{ik} \mathring{R}_{jk} \
		\end{multline}
	\end{corollary}

	\section{Proof of \thref{th:Io_1}, \thref{th:Io_2}}
	
	Let $ M^{n} $ closed smooth n-manifold with $ n \geq 3 $, let $ g $ be a critical metric for $ \mathfrak{F}_{t} $ on $ \mathcal{M}_{1}(M^{n}) $ for a generic $ t \in \mathbb{R} $, then, by \propref{th:Cat_5.16}:
	\begin{equation}\label{eq:proofGenEL}
		\mmmint{|\nabla \mathring{Ric}|^{2} - \frac{(n-2)(1 + 2t)}{2n} |\nabla R|^{2}} = 2 \mmmint{R_{ikjl} \mathring{R}_{ij} \mathring{R}_{kl} + \frac{1 + nt}{n} R |\mathring{Ric}|^{2}}
	\end{equation}
	Inspired by the proof of \cite[Theorem 1.1]{catino_rigidity_2019}, given $ a_{1} $, $ a_{2} $, $ b_{1} $, $ b_{2} $, $ b_{3} \in \mathbb{R} $ we introduce the 3-tensor $ F $ as follows:
	\[ F_{ijk} = \nabla_{k} \mathring{R}_{ij} + a_{1} \nabla_{j} \mathring{R}_{ik} + a_{2} \nabla_{i} \mathring{R}_{jk} + b_{1} \nabla_{k} R g_{ij} + b_{2} \nabla_{j} R g_{ik} + b_{3} \nabla_{i} R g_{jk} \]
	Its modulus, due to the contracted Bianchi identity $ div \mathring{Ric} = \frac{n-2}{2n} \nabla R $, reads:
	\begin{equation}\label{eq:proofGenF}
		|F|^{2} = (1 + a_{1}^{2} + a_{2}^{2}) |\nabla \mathring{Ric}|^{2} + 2 (a_{1} + a_{2} + a_{1} a_{2}) \nabla_{k} \mathring{R}_{ij} \nabla_{j} \mathring{R}_{ik} + Q_{0} |\nabla R|^{2}
	\end{equation}
	with:
	\[ Q_{0} := \frac{n-2}{n} [a_{1} (b_{1} + b_{3}) + a_{2} (b_{1} + b_{2}) + b_{2} + b_{3}] + n (b_{1}^{2} + b_{2}^{2} + b_{3}^{2}) + 2 (b_{1} b_{2} + b_{1} b_{3} + b_{2} b_{3}) \]
	By integrating \eqref{eq:proofGenF} over $ M $ and by isolating $ |\nabla \mathring{Ric}|^{2} $, we get:
	\begin{multline*}
		\mint{|\nabla \mathring{Ric}|^{2}} = \frac{1}{1 + a_{1}^{2} + a_{2}^{2}} \mint{|F|^{2}} - \frac{2 (a_{1} + a_{2} + a_{1} a_{2})}{1 + a_{1}^{2} + a_{2}^{2}} \mint{\nabla_{k} \mathring{R}_{ij} \nabla_{j} \mathring{R}_{ik}} \\ - \frac{Q_{0}}{1 + a_{1}^{2} + a_{2}^{2}} \mint{|\nabla R|^{2}}
	\end{multline*}
	We manipulate the second term on the right-hand side via integration by parts and via the covariant derivative commutation formula for $ \mathring{Ric} $ (again using the contracted Bianchi identity):
	\begin{align*}
		\mint{\nabla_{k} \mathring{R}_{ij} \nabla_{j} \mathring{R}_{ik}} &= - \mint{\mathring{R}_{ij} \nabla_{k} \nabla_{j} \mathring{R}_{ik}} \\
		&= - \mint{\mathring{R}_{ij} (\nabla_{j} \nabla_{k} \mathring{R}_{ik} + R_{kjil} \mathring{R}_{kl} + R_{jl} \mathring{R}_{il})} \\
		&= - \mmmint{\frac{n-2}{2n} \mathring{R}_{ij} \nabla_{j} \nabla_{i} R + R_{kjil} \mathring{R}_{ij} \mathring{R}_{kl} + \mathring{R}_{ij} \mathring{R}_{jl} \mathring{R}_{il} + \frac{R}{n} \mathring{R}_{ij} g_{jl} \mathring{R}_{il}} \\
		&= \mmmint{\left( \frac{n-2}{2n} \right)^{2} |\nabla R|^{2} + R_{jkil} \mathring{R}_{ij} \mathring{R}_{kl} - \mathring{R}_{ij} \mathring{R}_{il} \mathring{R}_{jl} - \frac{1}{n} R |\mathring{Ric}|^{2}}
	\end{align*}
	Thus, overall:
	\begin{multline*}
		\mint{|\nabla \mathring{Ric}|^{2}} = \frac{1}{1 + a_{1}^{2} + a_{2}^{2}} \mint{|F|^{2}} - Q_{1} \mint{|\nabla R|^{2}} \\
		- \frac{2(a_{1} + a_{2} + a_{1} a_{2})}{1 + a_{1}^{2} + a_{2}^{2}} \mint{(R_{jkil} \mathring{R}_{ij} \mathring{R}_{kl} - \mathring{R}_{ij} \mathring{R}_{il} \mathring{R}_{jl})} + \frac{2(a_{1} + a_{2} + a_{1} a_{2})}{n(1 + a_{1}^{2} + a_{2}^{2})} \mint{R |\mathring{Ric}|^{2}}
	\end{multline*}
	with:
	\[ Q_{1} := \frac{Q_{0}}{1 + a_{1}^{2} + a_{2}^{2}} + \left(\frac{n-2}{2n}\right)^{2} \frac{2 (a_{1} + a_{2} + a_{1} a_{2})}{1 + a_{1}^{2} + a_{2}^{2}} \]
	By pluggin this back into \eqref{eq:proofGenEL}, we get:
	\begin{align*}
		\frac{1}{1 + a_{1}^{2} + a_{2}^{2}} \mint{|F|^{2}} &- \left[Q_{1} + \frac{(n-2)(1+2t)}{2n}\right] \mint{|\nabla R|^{2}} \\
		&+ \frac{2(a_{1} + a_{2} + a_{1} a_{2})}{1 + a_{1}^{2} + a_{2}^{2}} \mint{\mathring{R}_{ij} \mathring{R}_{il} \mathring{R}_{jl}} \\
		&+ \left[\frac{2(a_{1} + a_{2} + a_{1} a_{2})}{n(1 + a_{1}^{2} + a_{2}^{2})} - \frac{2(1+nt)}{n}\right] \mint{R |\mathring{Ric}|^{2}} \\ 
		&- \frac{2(a_{1} + a_{2} + a_{1} a_{2})}{1 + a_{1}^{2} + a_{2}^{2}} \mint{R_{jkil} \mathring{R}_{ij} \mathring{R}_{kl}} - 2 \mint{R_{ikjl} \mathring{R}_{ij} \mathring{R}_{kl}} = 0
	\end{align*}
	Since:
	\[ R_{jkil} \mathring{R}_{ij} \mathring{R}_{kl} = R_{ikjl} \mathring{R}_{ji} \mathring{R}_{kl} = R_{ikjl} \mathring{R}_{ij} \mathring{R}_{kl} \]
	We get:
	\begin{align*}
		\frac{1}{1 + a_{1}^{2} + a_{2}^{2}} \mint{|F|^{2}} &- \left[Q_{1} + \frac{(n-2)(1+2t)}{2n}\right] \mint{|\nabla R|^{2}} \\
		&+ \frac{2(a_{1} + a_{2} + a_{1} a_{2})}{1 + a_{1}^{2} + a_{2}^{2}} \mint{\mathring{R}_{ij} \mathring{R}_{il} \mathring{R}_{jl}} \\
		&+ \left[\frac{2(a_{1} + a_{2} + a_{1} a_{2})}{n(1 + a_{1}^{2} + a_{2}^{2})} - \frac{2(1+nt)}{n}\right] \mint{R |\mathring{Ric}|^{2}} \\ 
		&- \left[\frac{2(a_{1} + a_{2} + a_{1} a_{2})}{1 + a_{1}^{2} + a_{2}^{2}} + 2\right] \mint{R_{ikjl} \mathring{R}_{ij} \mathring{R}_{kl}} = 0
	\end{align*}
	By assuming $ Sec \geq \epsilon R $, for some $ \epsilon \in \mathbb{R} $, by \eqref{eq:sec} we have:
	\begin{multline*}
		R_{ijkl} \mathring{R}_{ik} \mathring{R}_{jl} \leq \left( \frac{n^{2} - 4n + 2 - n^{2} (n-2) (n-3) \epsilon}{2n} - \frac{n-4}{2} (1 - n(n-1) \epsilon) s \right) R |\mathring{Ric}|^{2} \\ - (n - 1 - ns) \mathring{R}_{ij} \mathring{R}_{il} \mathring{R}_{jl}
	\end{multline*}
	\footnotetext{($ \dagger $) $ \frac{2(a_{1} + a_{2} + a_{1} a_{2})}{1 + a_{1}^{2} + a_{2}^{2}} + 2 \geq 0 \iff \frac{2(a_{1} + a_{2} + a_{1} a_{2} + 1 + a_{1}^{2} + a_{2}^{2})}{1 + a_{1}^{2} + a_{2}^{2}} \geq 0 \iff f(a_{1},a_{2}) := a_{1} + a_{2} + a_{1} a_{2} + 1 + a_{1}^{2} + a_{2}^{2} \geq 0 $, which holds, indeed:
		$ (\partial_{a_{1}} f = 1 + a_{2} + 2 a_{1} = 0 \ \land \ \partial_{a_{2}} f = 1 + a_{1} + 2 a_{2} = 0) \iff a_{1} = a_{2} = - \frac{1}{3} $ and: $ (\partial_{a_{1}a_{1}}^{2} f = \partial_{a_{2}a_{2}}^{2} f = 2 \ \land \ \partial_{a_{1}a_{2}}^{2} f = 1) \implies (\partial_{a_{1}a_{1}}^{2} f) (\partial_{a_{2}a_{2}}^{2} f) - (\partial_{a_{1}a_{2}}^{2} f)^{2} = 3 > 0 $, thus $ (-\frac{1}{3},-\frac{1}{3}) $ is the global minimum of $ f $ and $ f(-\frac{1}{3},-\frac{1}{3}) = \frac{2}{3} > 0 $. Thus we can use \eqref{eq:proofGenSec} with the correct sign without assumptions on $ a_{1} $ and $ a_{2} $.}
	\begin{multline}\label{eq:proofGenSec}
		\iff - R_{ijkl} \mathring{R}_{ik} \mathring{R}_{jl} \geq - \left( \frac{n^{2} - 4n + 2 - n^{2} (n-2) (n-3) \epsilon}{2n} - \frac{n-4}{2} (1 - n(n-1) \epsilon) s \right) R |\mathring{Ric}|^{2} \\ + (n - 1 - ns) \mathring{R}_{ij} \mathring{R}_{il} \mathring{R}_{jl} \
	\end{multline}
	By using it, since $ R_{ikjl} \mathring{R}_{ij} \mathring{R}_{kl} = R_{ijkl} \mathring{R}_{ik} \mathring{R}_{jl} $, we obtain:
	\begin{align*}
		0 &\overset{(\dagger)}{\geq} \frac{1}{1 + a_{1}^{2} + a_{2}^{2}} \mint{|F|^{2}}
		\begin{aligned}[t]
			&- \left[Q_{1} + \frac{(n-2)(1+2t)}{2n}\right] \mint{|\nabla R|^{2}} \\
			&+ \frac{2(a_{1} + a_{2} + a_{1} a_{2})}{1 + a_{1}^{2} + a_{2}^{2}} \mint{\mathring{R}_{ij} \mathring{R}_{il} \mathring{R}_{jl}} \\
			&+ \left[\frac{2(a_{1} + a_{2} + a_{1} a_{2})}{n(1 + a_{1}^{2} + a_{2}^{2})} - \frac{2(1+nt)}{n}\right] \mint{R |\mathring{Ric}|^{2}}
		\end{aligned} \\
		&- \left[\frac{2(a_{1} + a_{2} + a_{1} a_{2})}{1 + a_{1}^{2} + a_{2}^{2}} + 2\right] \left[ \frac{n^{2} - 4n + 2 - n^{2} (n-2) (n-3) \epsilon}{2n} - \frac{n-4}{2} (1 - n(n-1) \epsilon) s \right] \mint{R |\mathring{Ric}|^{2}} \\
		&+ \left[\frac{2(a_{1} + a_{2} + a_{1} a_{2})}{1 + a_{1}^{2} + a_{2}^{2}} + 2\right] (n - 1 - ns) \mint{\mathring{R}_{ij} \mathring{R}_{il} \mathring{R}_{jl}}
	\end{align*}\
	We define the following coefficients:
	\[ Q_{2} := Q_{1} + \frac{(n-2)(1+2t)}{2n} \]\
	\begin{align*}
		Q_{RRc} := &\left[\frac{2(a_{1} + a_{2} + a_{1} a_{2})}{n(1 + a_{1}^{2} + a_{2}^{2})} - \frac{2(1+nt)}{n}\right] \\
		- &\left[\frac{2(a_{1} + a_{2} + a_{1} a_{2})}{1 + a_{1}^{2} + a_{2}^{2}} + 2\right] \left[ \frac{n^{2} - 4n + 2 - n^{2} (n-2) (n-3) \epsilon}{2n} - \frac{n-4}{2} (1 - n(n-1) \epsilon) s \right]
	\end{align*} \\
	\begin{align*}
		Q_{3Rc} := &\frac{2(a_{1} + a_{2} + a_{1} a_{2})}{1 + a_{1}^{2} + a_{2}^{2}} + \left[\frac{2(a_{1} + a_{2} + a_{1} a_{2})}{1 + a_{1}^{2} + a_{2}^{2}} + 2\right] (n - 1 - ns) \\
		=&\frac{2(a_{1} + a_{2} + a_{1} a_{2})}{1 + a_{1}^{2} + a_{2}^{2}} (1 + n - 1 - ns) + 2 (n - 1 - ns) \\
		=&\frac{2(a_{1} + a_{2} + a_{1} a_{2})}{1 + a_{1}^{2} + a_{2}^{2}} n (1 - s) + 2 (n - 1 - ns)
	\end{align*}
	Thus:
	\[ 0 \geq \frac{1}{1 + a_{1}^{2} + a_{2}^{2}} \mint{|F|^{2}} - Q_{2} \mint{|\nabla R|^{2}} + Q_{RRc} \mint{R |\mathring{Ric}|^{2}} + Q_{3Rc} \mint{\mathring{R}_{ij} \mathring{R}_{il} \mathring{R}_{jl}} \]
	Now we look for $ s \in \mathbb{R} $ such that $ Q_{3Rc} = 0 $ and we verify that $ s \in [0,1] $:
	\begin{align}
		Q_{3Rc} = 0 &\iff \frac{2(a_{1} + a_{2} + a_{1} a_{2})}{1 + a_{1}^{2} + a_{2}^{2}} n (1 - s) + 2 (n - 1 - ns) = 0 \nonumber \\
		&\iff 2 n s \left(1 + \frac{a_{1} + a_{2} + a_{1} a_{2}}{1 + a_{1}^{2} + a_{2}^{2}}\right) = \frac{2(a_{1} + a_{2} + a_{1} a_{2})}{1 + a_{1}^{2} + a_{2}^{2}} n + 2 (n - 1) \nonumber \\
		&\iff 2 n s \left(1 + \frac{a_{1} + a_{2} + a_{1} a_{2}}{1 + a_{1}^{2} + a_{2}^{2}}\right) = 2 n \left(1 + \frac{a_{1} + a_{2} + a_{1} a_{2}}{1 + a_{1}^{2} + a_{2}^{2}}\right) - 2 \nonumber \\
		&\iff n s \left(1 + \frac{a_{1} + a_{2} + a_{1} a_{2}}{1 + a_{1}^{2} + a_{2}^{2}}\right) = n \left(1 + \frac{a_{1} + a_{2} + a_{1} a_{2}}{1 + a_{1}^{2} + a_{2}^{2}}\right) - 1 \nonumber \\
		&\iff s = \underbrace{\frac{n \left(1 + \frac{a_{1} + a_{2} + a_{1} a_{2}}{1 + a_{1}^{2} + a_{2}^{2}}\right) - 1}{n \left(1 + \frac{a_{1} + a_{2} + a_{1} a_{2}}{1 + a_{1}^{2} + a_{2}^{2}}\right)}}_{\implies s < 1} \label{eq:proofGenS}
	\end{align}
	Hence:
	\begin{align*}
		s \geq 0 &\overset{f \geq 0}{\iff} n \left(1 + \frac{a_{1} + a_{2} + a_{1} a_{2}}{1 + a_{1}^{2} + a_{2}^{2}}\right) - 1 \geq 0 \\
		&\iff 1 + \frac{a_{1} + a_{2} + a_{1} a_{2}}{1 + a_{1}^{2} + a_{2}^{2}} - \frac{1}{n} \geq 0 \\
		&\iff \frac{n-1}{n} + \frac{a_{1} + a_{2} + a_{1} a_{2}}{1 + a_{1}^{2} + a_{2}^{2}} \geq 0 \\
		&\iff \frac{(n-1) (1 + a_{1}^{2} + a_{2}^{2}) + n (a_{1} + a_{2} + a_{1} a_{2})}{n (1 + a_{1}^{2} + a_{2}^{2})} \geq 0 \\
		&\iff (n-1) (1 + a_{1}^{2} + a_{2}^{2}) + n (a_{1} + a_{2} + a_{1} a_{2}) \geq 0 \\
		&\iff g(a_{1},a_{2}) := (n-1) f(a_{1},a_{2}) + a_{1} + a_{2} + a_{1} a_{2} \geq 0 
	\end{align*}
	Which is true, indeed:
	\begin{equation*}
		\begin{cases}  
			\partial_{a_{1}} g = (n-1) \partial_{a_{1}} f + 1 + a_{2} = (n-1) + (n-1) a_{2} + 2 (n-1) a_{1} + 1 + a_{2} = 0 \\
			\partial_{a_{2}} g = (n-1) \partial_{a_{2}} f + 1 + a_{1} = (n-1) + (n-1) a_{1} + 2 (n-1) a_{2} + 1 + a_{1} = 0
		\end{cases}
	\end{equation*}
	\[ \iff a_{1} = a_{2} = - \frac{n}{3n-2} \]
	and, since:
	\begin{equation*}
		\begin{cases}
			\partial_{a_{1}a_{1}}^{2} g = (n-1) \partial_{a_{1}a_{1}}^{2} f = 2 (n-1) \\
			\partial_{a_{2}a_{2}}^{2} g = (n-1) \partial_{a_{2}a_{2}}^{2} f = 2 (n-1) \\
			\partial_{a_{1}a_{2}}^{2} g = (n-1) \partial_{a_{1}a_{2}}^{2} f + 1 = n
		\end{cases}
	\end{equation*}
	we get:
	\[ (\partial_{a_{1}a_{1}}^{2} g) (\partial_{a_{2}a_{2}}^{2} g) - (\partial_{a_{1}a_{2}}^{2} g)^{2} = 4 (n-1)^{2} - n^{2} = 3 n^{2} - 8 n + 4 > 0 \]
	i.e. $ (-\frac{n}{3n-2},-\frac{n}{3n-2}) $ is the global minimum of $ g $ and, since:
	\begin{align*}
		g\left(-\frac{n}{3n-2}\, ,-\frac{n}{3n-2}\right) &= (n-1) f\left(-\frac{n}{3n-2}\, ,-\frac{n}{3n-2}\right) - \frac{2n}{3n-2} + \frac{n^{2}}{(3n-2)^{2}} \\
		&= (n-1) \left[ 1 - \frac{2n}{3n-2} + \frac{3n^{2}}{(3n-2)^{2}} \right] - \frac{2n}{3n-2} + \frac{n^{2}}{(3n-2)^{2}} \\
		&= (n-1) - (n-1) \frac{2n}{3n-2} + (n-1) \frac{3n^{2}}{(3n-2)^{2}} - \frac{2n}{3n-2} + \frac{n^{2}}{(3n-2)^{2}} \\
		&=\frac{2 n^{2} - 5 n + 2}{3 n - 2} \overset{n \geq 3}{>} 0
	\end{align*}
	we get $ s > 0 $; i.e. $ s \in (0,1) $.
	\\By substituting \eqref{eq:proofGenS} into $ Q_{RRc} $, and by collecting $ \epsilon $, we get:
	\begin{equation*}
		Q_{RRc} = - 1 - 2 t + \left( \frac{2 n (a_{1} + a_{2} + a_{1} a_{2})}{1 + a_{1}^{2} + a_{2}^{2}} + 4 + n (n-3) \right) \epsilon
	\end{equation*}
	To procede with the argument, either we assume $ R = const $ or we force $ Q_{2} = 0 $. For the moment being we procede with the former, obtaining:
	\[ 0 \geq \frac{1}{1 + a_{1}^{2} + a_{2}^{2}} \mint{|F|^{2}} + Q_{RRc} \mint{R |\mathring{Ric}|^{2}} \]
	Since $ R \geq 0 $ by assumption, we ask for:
	\[ Q_{RRc} \geq 0 \overset{(\dagger)}{\iff} \epsilon \geq \underbrace{\frac{1+2t}{\frac{2 n (a_{1} + a_{2} + a_{1} a_{2})}{1 + a_{1}^{2} + a_{2}^{2}} + 4 + n (n-3)}}_{=: \, \texttt{eps}} \]
	\footnotetext{($ \dagger $) Since $ \frac{2 n (a_{1} + a_{2} + a_{1} a_{2})}{1 + a_{1}^{2} + a_{2}^{2}} + 4 + n (n-3) > 0 \quad \forall a_{1},a_{2} \in \mathbb{R}, n \geq 3 $ \\}
	Through \texttt{Mathematica} (see \hyperref[ch:Code]{ Appendix A} for the code), we minimize \texttt{eps[a1,a2]} with respect to our degrees of freedom $ a_{1} $ and $ a_{2} $, obtaining:
	\begin{equation*}
		\epsilon \geq \begin{cases}
			\frac{1 + 2 t}{(n-2)^{2}} \quad &t < - \frac{1}{2} \\
			0 \quad &t = - \frac{1}{2} \\
			\frac{1 + 2 t}{n^{2} - n + 4} \quad &t > - \frac{1}{2}
		\end{cases}
	\end{equation*}
	Which is entailed by:
	\begin{equation}\label{eq:choices}
		a_{1} = \begin{cases}
			-2 \quad &t < - \frac{1}{2} \\
			0 \quad &t = - \frac{1}{2} \\
			1 \quad &t > - \frac{1}{2}
		\end{cases}
		\qquad\qquad
		a_{2} = \begin{cases}
			1 \quad &t \neq - \frac{1}{2} \\
			0 \quad &t = - \frac{1}{2}
		\end{cases}
	\end{equation}
	Finally by selecting the strict inequality, we force $ \mathring{Ric} = 0 $, i.e. $ g $ is Einstein.
	\\Now we distinguish between the cases $ t \leq - \frac{1}{2} $ and $ t > - \frac{1}{2} $. For the former, the optimal choices \eqref{eq:choices} are actually compatible with the requirement of $ Q_{2} = 0 $. Infact, again through \texttt{Mathematica}, we satisfy it by picking:
	\[ b_{1} = \frac{1}{n}+\frac{1}{2-2 n} - \frac{1}{n^2} \sqrt{\frac{3 n^2 (n-2) [4t - (1 + 4t)n]}{2(n-1)}} \qquad b_{2} = - \frac{n - 2}{n(n-1)} \qquad b_{3} = \frac{n^2-n-2}{2 n (n+1) (n-1)} \]
	for $ t < - \frac{1}{2} $. And by picking:
	\[
	b_{1} = \frac{1}{n} - \frac{1}{2 n} \sqrt{\frac{2(n-2)^2}{n^2+n-2}} - \frac{n}{n^2+n-2} \qquad b_{2} = b_{3} = -\frac{n-2}{2 (n^2+n-2)}
	\]
	for $ t = - \frac{1}{2} $. That is, there is no need of requiring $ R = const $ for any $ n \geq 3 $ and for any $ t \leq - \frac{1}{2} $. This concludes the proof of \thref{th:Io_1}.
	\\The case $ t > - \frac{1}{2} $ is more troubled. Indeed, the optimal choices \eqref{eq:choices} are incompatible with the requirement of $ Q_{2} = 0 $, since they yield complex solutions for any choice of $ b_{1} $, $ b_{2} $, $ b_{3} $, $ t > - \frac{1}{2} $ and $ n \geq 3 $. Hence, for $ t > - \frac{1}{2} $, we must assume $ R = const $ in order to keep the optimal bound on $ \epsilon $. This concludes the proof of \thref{th:Io_2}.
	\hfill \qedsymbol
	\vspace{0.83cm}

	\addtocontents{toc}{\vspace{2em}}
	\printbibliography
	
	\pagestyle{empty}
	\includepdf[pages=1,pagecommand={\section*{Appendix A} \label{ch:Code}}]{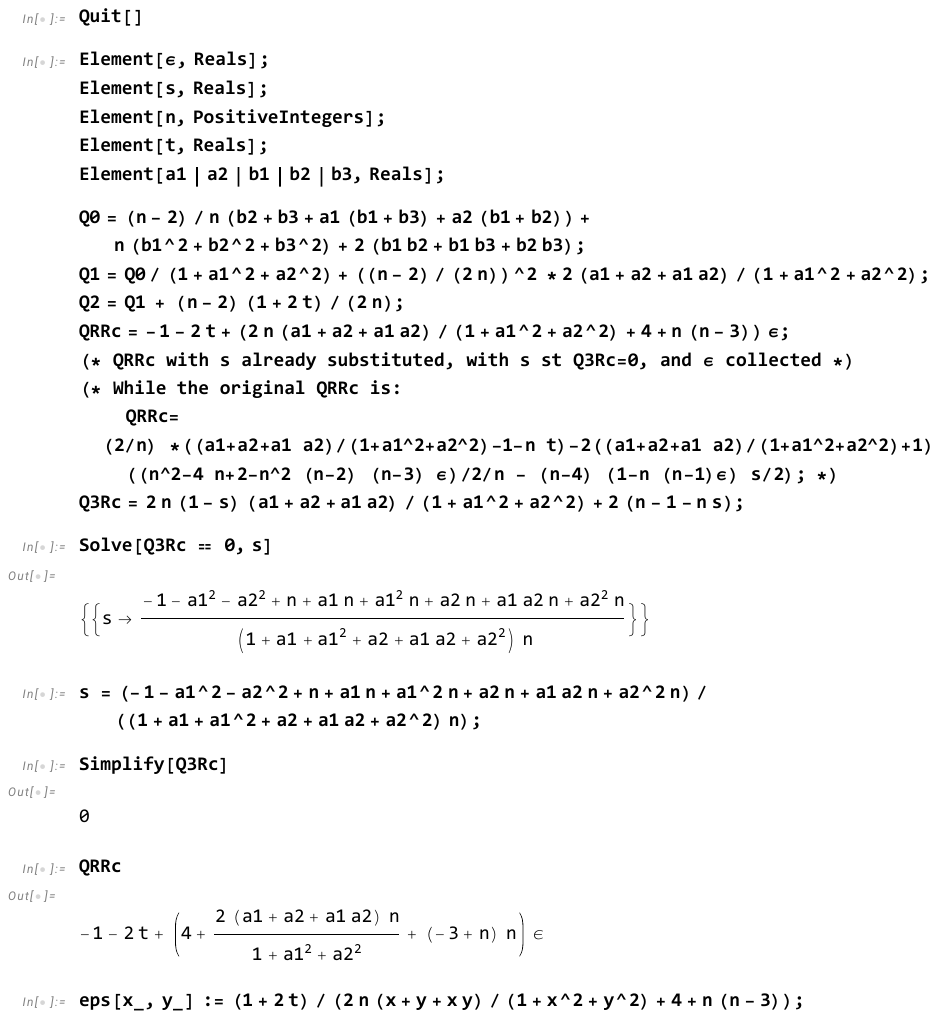}
	\includepdf[pages=2-]{code.pdf}
	
\end{document}